\documentclass[12pt,a4paper,reqno]{amsart}
\usepackage{amsmath}
\usepackage{amsfonts}
\usepackage{amssymb,amsthm,amsfonts,amsthm,latexsym,enumerate,url,cases}
\numberwithin{equation}{section}
\usepackage{mathrsfs}
\newtheorem{theorem*}{Theorem}
\newtheorem{lemma*}{Lemma}
\theoremstyle{plain}
\newtheorem{theorem}{Theorem}
\newtheorem{lemma}[theorem]{Lemma}
\newtheorem{corollary}[theorem]{Corollary}

\theoremstyle{definition}


\begin{document}

\title
[{On a conjecture of Erd\H{o}s}] {On a conjecture of Erd\H{o}s}

\author
[Yong-Gao Chen and Yuchen Ding] {Yong-Gao Chen and Yuchen Ding*}

\address{(Yong-Gao Chen) School of Mathematical Science,
Nanjing Normal University, Nanjing 210023, People's Republic of
China} \email{ygchen@njnu.edu.cn}
\address{(Yuchen Ding) School of Mathematical Science,  Yangzhou University, Yangzhou 225002, People's Republic of China}
\email{ycding@yzu.edu.cn}
\thanks{*Corresponding author}


\begin{abstract}  Let $\mathcal{P}$ denote the set of all primes.  In
1950, P. Erd\H os conjectured that if $c$ is an arbitrarily given
constant, $x$ is sufficiently large and $a_1,\dots , a_t$ are
positive integers with
$$a_1<a_2<\cdot\cdot\cdot<a_t\leqslant x,~t>\log x,$$
then there exists an integer $n$ so that the number of solutions
of $n=p+a_i$ $(p\in \mathcal{P}, 1\le i\le t)$ is greater than
$c$. In this note, we confirm this old conjecture of Erd\H{o}s.
\end{abstract}
\maketitle

\baselineskip 18pt

\section{Introduction}

Let $\mathcal{P}$ denote the set of all primes.  In 1950,
Erd\H{o}s \cite{Er1} made the following anecdotal conjecture:

{\bf Erd\H os Conjecture.} {\it Let $c$ be any constant and $x$
sufficiently large,
$$a_1<a_2<\cdot\cdot\cdot<a_t\leqslant x,~t>\log x.$$
Then there exists an integer $n$ so that the number of solutions
of $n=p+a_i$ $(p\in \mathcal{P}, 1\le i\le t)$ is greater than
$c$.}

 Erd\H{o}s \cite{Er1} himself proved this conjecture for the case $a_i=2^i$, which gives
an affirmative answer to a question of Tur\'{a}n. In a former note
\cite{Di}, the second author proved this conjecture for the case
$a_i\mid a_{i+1}$ with its quantitative form, which is a slight
generalization of Erd\H{o}s' result. In a subsequent note, the
second author and Zhou \cite{DZ} proved the conjecture for the
case $a_i=2^{p_i}$, where $p_i$ is the $i$-th prime. This case was
conjectured by the first author \cite{Ch} years ago. Shortly
after, the authors of the present note recognized that the
complete proof of Erd\H{o}s' conjecture actually follows directly
from a new achievement of the distributions of the primes
established by Maynard--Tao \cite{Ma,Ta}. We keep record here as
the closure of this longstanding conjecture.

In this note, the following general results are proved. The Erd\H
os Conjecture follows from Corollary \ref{erdosconjecture}.

\begin{theorem}\label{thm1} For any $\ell$
distinct integers $a_1, \dots , a_\ell $, there are infinitely
many positive integers $n$ such that the number of solutions of
$n=p+a_i$ $(p\in \mathcal{P}, 1\le i\le \ell )$ is greater than
$$ \frac 18 \log \ell -1.6.$$
\end{theorem}

From Theorem \ref{thm1}, we immediately have the following
corollaries:

\begin{corollary} \label{erdosconjecture} Let  $x\ge 2$ and
$$a_1<a_2<\cdot\cdot\cdot<a_t\leqslant x,~t>\log x.$$
Then there exists infinitely many integers $n$ so that the number
of solutions of $n=p+a_i$ $(p\in \mathcal{P}, 1\le i\le t)$ is
greater than $$ \frac 18 \log \log x -1.6.$$
\end{corollary}

\begin{corollary} \label{infinitecor} Let $\mathcal{A}=\{a_i\}_{i=1}^{\infty}$ be an infinite
set of integers and let
$$f_\mathcal{A}(n)=\#\{(p,a):n=p+a,p\in\mathcal{P},
a\in\mathcal{A}\} ,$$ then $$\limsup_{n\to+\infty}
f_\mathcal{A}(n)=+\infty .$$
\end{corollary}

\section{Proofs}

 A set $\{ b_1,\dots , b_k\}$ is
called an admissible set if there is no a fixed integer $d>1$ such
that $d\mid (n+b_1)\cdots (n+b_k)$ for all integers $n$. It is
equivalent that $\{ b_1,\dots , b_k\}$ modulo $p$ occupies at most
$p-1$ residues. We begin with the following deep result for the
distribution of the primes due to Maynard--Tao.

\begin{lemma}\label{MT}\cite[Theorem 6.2]{Gr} For any given integer $m\geqslant 2$, let $k$ be a positive integer
with $k\log k > e^{8m+4}$. For any admissible set $\{ b_1,\dots ,
b_k\}$, there are infinitely many integers $n$ such that at least
$m$ of $n+b_1,\dots ,n+b_k$ are prime numbers.
\end{lemma}

\begin{lemma}\cite[Lemma 3]{ChSun}\label{Mertens} We have
$$\prod_{3\le p\le x} \left( 1-\frac 1p\right)^{-1}\le 0.923 \log x, \quad x\ge 74,$$
where the product is taken over all primes $p$ with $3\le p\le x$.
\end{lemma}

\begin{proof}[Proof of Theorem \ref{thm1}] If $\ell \le e^{12}$, then
$$\frac 18 \log \ell -1.6\le 0,$$
and Theorem \ref{thm1} is trivial. In the following, we assume
that $\ell >e^{12}$.

Let $p_i$ be the $i$-th prime. Assume that $a_1, \dots , a_\ell $
are $\ell$ distinct integers. For $p_1$, one of residues modulo
$p_1$ contains at most $\lfloor \ell /p_1\rfloor $ of $a_1, \dots
, a_\ell $. So at least $\ell -\lfloor \ell /p_1\rfloor $ of $a_1,
\dots , a_\ell $ occupy at most $p_1-1$ residues  modulo $p_1$.
Let $\ell_0=\ell$ and $\ell_1=\ell -\lfloor \ell /p_1\rfloor $.
Without loss of generality, we assume that $a_1, \dots ,
a_{\ell_1} $ occupy at most $p_1-1$ residues modulo $p_1$.
Similarly,  without loss of generality, we may assume that $a_1,
\dots , a_{\ell_2} $ occupy at most $p_2-1$ residues  modulo
$p_2$, where  $\ell_2=\ell_1 -\lfloor \ell_1 /p_2\rfloor $.
Continuing this process, at the $t$-th step, we may assume that
$a_1, \dots , a_{\ell_t} $ occupy at most $p_t-1$ residues  modulo
$p_t$, where  $\ell_t=\ell_{t-1} -\lfloor \ell_{t-1} /p_t\rfloor
$. Since $\ell \ge \ell_1 \ge \cdots $ and $p_1<p_2<\cdots $,
there exists $t$ with $\ell_t <p_{t+1}$. Let $s$ be the least
integer with $\ell_s <p_{s+1}$. It is clear that $\{ a_1, \dots ,
a_{\ell_s}\}$ is an admissible set. Let $m$ be largest integer
with $\ell_s\log \ell_s
> e^{8m+4}$.  If $m\ge 2$, then by Lemma \ref{MT}, there are infinitely many
integers $n$ such that at least $m$ of $n-a_{1},\dots ,
n-a_{\ell_s}$ are prime numbers. Since there are infinitely many
primes, it follows that there are infinitely many integers $n$
such that at least one of $n-a_{1},\dots , n-a_{\ell_s}$ is prime
number. So the conclusion is also true for $m\le 1$.

Now we establish an explicit relation between $\ell$ and $m$.

Since $$\ell_{i+1}=\ell_i - \left\lfloor \frac{\ell_i
}{p_{i+1}}\right\rfloor \ge \ell_i - \frac{\ell_i
}{p_{i+1}}=\ell_i \left( 1-\frac 1{p_{i+1}}\right),\quad i=0, 1,
\dots ,$$ it follows from the definition of $s$ that
\begin{eqnarray}\label{eq1}p_{s+1}&>&\ell_s\ge \ell_{s-1}\left( 1 - \frac
1{p_{s}}\right)\ge \cdots  \nonumber \\
&\ge & \ell \left( 1-\frac 1{p_1}\right)
\cdots \left( 1-\frac 1{p_s}\right)\nonumber \\
&>&e^{12} \left( 1-\frac 1{p_1}\right) \cdots \left( 1-\frac
1{p_s}\right).\end{eqnarray} Noting that
$$p_{101}<e^{12} \left( 1-\frac 1{p_1}\right) \cdots \left( 1-\frac
1{p_{100}}\right),$$ for $1\le i\le 100$ we have
\begin{eqnarray*}p_{i+1}\le  p_{101}&<&e^{12} \left( 1-\frac 1{p_1}\right) \cdots \left( 1-\frac
1{p_{100}}\right)\\
&\le &  e^{12} \left( 1-\frac 1{p_1}\right) \cdots \left( 1-\frac
1{p_{i}}\right).\end{eqnarray*} It follows from \eqref{eq1} that
$s> 100$. So $p_s\ge p_{101}=547$. Thus, by Lemma \ref{Mertens},
$$\ell_s\ge  \ell \left( 1-\frac 1{p_1}\right) \cdots \left(  1-\frac
1{p_s}\right) \ge \frac \ell 2\cdot \frac{1}{0.923 \log
p_s}=\frac{\ell }{1.846 \log p_s}.$$ By the definition of $s$,
$p_s\le \ell_{s-1}$. Thus,
$$\ell_s \ge \left( 1-\frac 1{p_s}\right) \ell_{s-1}\ge  \left( 1-\frac 1{p_s}\right)
p_{s}=p_s-1.$$ It follows that
$$\ell_s \ge \frac{\ell }{1.846 \log p_s} \ge \frac{\log 546}{1.846 \log 547} \frac{\ell }{\log (p_s-1)}>\frac{0.54 \ell}{\log \ell_s}.$$
 So $\ell_s \log \ell_s \ge 0.54 \ell $. In view of the definition of $m$,
$$e^{8m+12}\ge \ell_s\log \ell_s \ge 0.54 \ell.$$
So
$$m\ge \frac 18 \log \ell -\frac{12}8+\frac{\log 0.54}8>\frac 18 \log \ell -1.6.$$

This completes the proof of Theorem \ref{thm1}.
\end{proof}

\begin{proof}[Proof of Corollary \ref{erdosconjecture}]
 Assume that
$$1\le a_1<\cdots <a_t\le x, \quad t>\log x.$$
By Theorem \ref{thm1}, there are infinitely many positive integers
$n$ such that the number of solutions of $n=p+a_i$ $(p\in
\mathcal{P}, 1\le i\le t )$ is greater than
$$\frac 18\log t-1.6>\frac 18 \log\log x-1.6.$$

This completes the proof of Corollary \ref{erdosconjecture}.
\end{proof}

\begin{proof}[Proof of Corollary \ref{infinitecor}]

By Theorem \ref{thm1}, there is a positive integers $n$ such that
the number of solutions of $n=p+a_i$ $(p\in \mathcal{P}, 1\le i\le
\ell )$ is greater than $\frac 18 \log \ell -1.6$. That is,
$f_\mathcal{A}(n)\ge \frac 18 \log \ell -1.6$. Now Corollary
\ref{infinitecor} follows immediately.
\end{proof}

\section{Remarks}

It is known that there is a positive proportion of positive odd
numbers  that  can be represented as $p+2^k$ with $k\in
\mathbb{N}$ and $p\in \mathcal{P}$ (See Romanoff \cite{Romanoff} )
and there is an arithmetical progression of positive odd numbers
none of which can be represented as $p+2^k$ with $k\in \mathbb{N}$
and $p\in \mathcal{P}$ (see Erd\H os \cite{Er1}).

Since one can take $k\le cm^2 e^{4m}$ for some positive constant
$c$ in Lemma \ref{MT} (see \cite{Gr} ), it follows that $\frac 18$
in Theorem \ref{thm1} and Corollary \ref{erdosconjecture} can be
improved to any constant less than $\frac 14$ for all sufficintly
large $\ell$ and $x$.

\section*{Acknowledgments}
The first author was supported by the National Natural Science
Foundation of China, Grant No. 12171243. The second author was
supported by the Natural Science Foundation of Jiangsu Province of
China, Grant No. BK20210784. He was also supported by the
foundations of the projects ``Jiangsu Provincial
Double--Innovation Doctor Program'', Grant No. JSSCBS20211023 and
``Golden  Phenix of the Green City--Yang Zhou'' to excellent PhD,
Grant No. YZLYJF2020PHD051.

\end{document}